\documentclass[11pt]{article}
\usepackage{a4}
\usepackage{color,graphicx,psfrag,bbm}
\usepackage{amssymb,amsmath,amsthm}

\DeclareMathOperator{\topint}{int}

\newcommand{\R}{{\bf R}}
\newcommand{\Q}{{\bf Q}}

\newcommand{\Z}{{\bf Z}}
\newcommand{\s}{{\bf S}}
\newcommand{\complex}[1]{\mathsf{#1}}
\newcommand{\K}{\complex{K}}

\newcommand{\obsr}{o_{\Z_2}^r}
\newcommand{\obsd}{o_{\Z_2}^d}
\newcommand{\obstwod}{o_{\Z_2}^{2d}}
\newcommand{\delprod}[1]{{#1}^{2}_{\textrm{del}}}

\DeclareMathOperator{\supp}{supp}
\DeclareMathOperator{\st}{st}
\DeclareMathOperator{\lk}{lk}

\makeatletter
\def\imod#1{\allowbreak\mkern10mu({\operator@font mod}\,\,#1)}
\makeatother

\theoremstyle{plain}

\newtheorem{theorem}{Theorem}[section]
\newtheorem{proposition}[theorem]{Proposition}
\newtheorem{lemma}[theorem]{Lemma}
\newtheorem*{lemma*}{Auxiliary Lemma}
\newtheorem{observation}[theorem]{Observation}

\newtheorem{conjecture}[theorem]{Conjecture}
\newtheorem{problem}[theorem]{Problem}

\theoremstyle{definition}

\theoremstyle{remark}

\newtheorem{claim}{Claim}
\newtheorem{remark}[theorem]{Remark}

\title{On the Embeddability of Skeleta of Spheres}
\author{Eran Nevo\footnote{Institute of Mathematics, The Hebrew
University, Jerusalem Israel, E-mail address: \texttt{
eranevo@math.huji.ac.il}} \ and Uli Wagner\footnote{Institut f\"{u}r Theoretische Informatik,
ETH Z\"{u}rich, CH-8092 Z\"{u}rich Switzerland, E-mail address: \texttt{uli@inf.ethz.ch}. Research supported by the Swiss National Science Foundation, SNF Project 200021-116741.}}
\begin{document}
\maketitle
\begin{abstract}
We consider a generalization of the van Kampen-Flores Theorem and
relate it to the long-standing $g$-conjecture for simplicial
spheres.
\end{abstract}
\section{Introduction and results}\label{Introduction}
A well-known result by van Kampen~\cite{vanKampen:KomplexeInEuklidischenRaeumen-1932, vanKampen:Berichtigung-1932} and by Flores \cite{Flores:NichtEinbettbar-1933} asserts that
\begin{theorem}\label{thmVK-F}
The $d$-skeleton of the $(2d+2)$-simplex does not embed in the $2d$-sphere.
\end{theorem}
The special case $d=1$, which says that the complete graph on 5 vertices is not planar, is particularly well-known. We remark that the proof of the theorem, which we review below, guarantees something stronger: for any continuous map from the $d$-skeleton of the $(2d+2)$-simplex into the $2d$-sphere $\s^{2d}$, there are two vertex-disjoint faces whose images intersect.
\medskip

Our main result is the following generalization of the van
Kampen-Flores Theorem to \emph{piecewise linear} (PL)
spheres (i.e. those that can be obtained by finitely many
\emph{bistellar moves}, also called \emph{Pachner moves}, from the boundary of a
simplex):
\begin{theorem}\label{thmMissingFacePL}
Let $\complex{S}_{\leq d}$ be the $d$-skeleton of a piecewise
linear 2d-sphere and let $M$ be a \emph{missing} $d$-face of
$\complex{S}_{\leq d}$, i.e. $\dim M=d$, $M\notin
\complex{S}_{\leq d}$ and the boundary $\partial M$ is contained
in $\complex{S}_{\leq d}$. Then the union $\K:=\complex{S}_{\leq
d} \cup\{M\}$ does not embed in the $2d$-sphere.
\end{theorem}

Our proof is based on an analysis of how the van Kampen
obstruction behaves under bistellar moves. However, we conjecture
that the condition of piecewise linearity is superfluous:
\begin{conjecture}\label{conjMissingFace}
Let $\complex{S}_{\leq d}$ be the $d$-skeleton of a triangulated
$2d$-sphere and let $M$ be a \emph{missing} $d$-face of
$\complex{S}_{\leq d}$. Then the union $\K:=\complex{S}_{\leq
d}\cup\{M\}$ does not embed in the $2d$-sphere.
\end{conjecture}

Observe that the case where the triangulated sphere is the one
obtained from the boundary of the $(2d+1)$-simplex by a stellar
subdivision at a $d$-face is exactly Theorem \ref{thmVK-F} (the
unique missing $d$-face is the one we subdivided). We remark that
for $d=1$, the conjecture is obviously true (and is also covered
by Theorem \ref{thmMissingFacePL}, as all triangulated $2$-spheres
are PL): The graphs (1-skeleta) of triangulations of $\s^2$ are
precisely the planar simple graphs with the maximal number of
$3n-6$ edges (by Euler's relation), where $n\geq 4$ is the number
of vertices; thus, by adding a missing edge, the graph becomes
nonplanar.
%, by simple counting.

One of the motivations for Conjecture~\ref{conjMissingFace} is
another conjecture, due to Kalai and
Sarkaria~\cite{Kalai:AlgebraicShifting-02}, that relates
embeddability of simplicial complexes to algebraic shifting and
that would imply the $g$-conjecture for simplicial spheres (see
\cite{McMullen:NumberFaces-71}). We explain these connections in
detail in Section~\ref{sec:g-conj} below. Theorem
\ref{thmMissingFacePL} can be seen as supporting evidence for
the Kalai-Sarkaria conjecture.

 For non-PL spheres,  the following weakening of Conjecture \ref{conjMissingFace} is true: $\K$ is not the $d$-skeleton of a triangulation of the $2d$-sphere. In fact, we can prove a bit more:

\begin{theorem}\label{thmNonSkeleton}
Let $\complex{S}$ be a triangulated $2d$-sphere and let $\complex{L}$ be an induced subcomplex of $\complex{S}$ which is homeomorphic to the $(d-1)$-sphere. Let $\complex{B}$ be a triangulated $d$-ball with boundary simplicially isomorphic to $\complex{L}$ by $g:\partial(\complex{B})\longrightarrow \complex{L}$. Then the glued complex $\complex{S}_{\leq d} \cup_g\complex{B}$ is not the $d$-skeleton of a triangulated $2d$-sphere.
\end{theorem}

We also consider the following second generalization of the van Kampen-Flores Theorem:
\begin{theorem}\label{thmHighSkel}
Let $\K$ be the $d$-skeleton of a triangulated $(2d+1)$-sphere. Then $\K$ does not embed in the $2d$-sphere.
\end{theorem}

Note that this generalization is obvious if the triangulated $(2d+1)$-sphere is piecewise linear. In this case $\K$ is a subdivision of the boundary of the $(2d+2)$-simplex and hence contains a subdivision of its $d$-skeleton.

It is quite natural to wonder if the following combinatorial strengthening of Theorems \ref{thmHighSkel} and \ref{thmMissingFacePL} is true:
\begin{problem}
Does any $\K$ as in Theorem \ref{thmHighSkel} or in Theorem \ref{thmMissingFacePL} contain a subdivision of the $d$-skeleton of the $(2d+2)$-simplex?
\end{problem}
\noindent We suspect that the answer is \emph{no} for $d>1$, at least in case of Theorem \ref{thmMissingFacePL}.

This paper is organized as follows: in Section \ref{sec:Background} we fix the notation and give the necessary background on van Kampen's obstruction for embeddablity. In Section \ref{sec:HighSkel} we prove Theorem \ref{thmHighSkel}, in Section \ref{sec:Dancis} we prove Theorem \ref{thmNonSkeleton}, in Section \ref{sec:bistellar} we prove Theorem \ref{thmMissingFacePL} by analyzing the effect of bistellar moves on van Kampen's obstruction, in Section \ref{sec:g-conj} we relate Theorem \ref{thmMissingFacePL} to the $g$-conjecture and to the conjecture of Kalai and Sarkaria.

\paragraph{Acknowledgements.} Part of this research was conducted during several visits of the second author to the Hebrew University of Jerusalem. He would like to thank the Einstein Institute of Mathematics and the School of Computer Science and Engineering, in particular Gil Kalai and Nati Linial, for their hospitality and support. Other part of this research was conducted during a visit of the first author to ETH; he thanks Emo Welzl and the Institute of Theoretical Computer Science at ETH for their hospitality and support.

\section{Background}\label{sec:Background}
Let $\K$ be a finite abstract simplicial complex. We denote by $\|\K\|$ the underlying topological space obtained by embedding $\K$ linearly into some Euclidean space. We use the notation $\K_{\leq i}:=\{\sigma \in \K: \dim \sigma \leq i\}$ for the $i$-skeleton of $\K$ and $\K_i :=\{\sigma \in \K: \dim \sigma = i\}$; in particular, $\K_0$ is the set of vertices.
%The \emph{link} of a face $\sigma$ in $\K$ is $\lk(\sigma,\K)=\{\rho \in \K: \rho \ast \sigma \in \K\}$; the (closed) \emph{star} is $\st(\sigma,\K):=\{\rho\ast \sigma: \rho \in \lk(\sigma,\K)\}$.

For a commutative ring $R$, we denote the reduced simplicial homology and cohomology of $\K$ with coefficients in $R$ by $\widetilde{H}_\ast(\K; R)$ and $\widetilde{H}^\ast(\K; R)$, respectively.
When working with deleted products of simplicial complexes, we will also need to consider slightly
more general complexes, namely polyhedral cell complexes, and we use the same notation for cellular (co)homology. We refer to Munkres' textbook~\cite{Munkres:AlgebraicTopology-1984} for general background on simplicial complexes and (co)homology.

\paragraph{Piecewise linear spheres and (bi)stellar moves.}

We briefly recall some basic notions and results from piecewise linear topology; our presentation follows
Lickorish's survey article \cite{Lickorish:SimplicialMoves-1999}, in which details and further references can be found. Consider two simplicial complexes $\K$ and $\K'$ linearly embedded in some Euclidean space $\R^d$. $\K'$ is called a \emph{subdivision} of $\K$ if $\|\K\|=\|\K'\|$ and each simplex of $\K'$ is fully contained in some simplex of $\K$.
A \emph{piecewise linear map} between two complexes $\K$ and $\complex{L}$ is a simplicial map from some subdivision of $\K$ to some subdivision of $\complex{L}$, and two complexes are \emph{piecewise linearly homeomorphic} if they have subdivisions that are simplicially isomorphic.

A \emph{piecewise linear $d$-ball} is a simplical complex that is piecewise linearly homeomorphic to $\Delta^d$, the complex consisting of the $d$-simplex and all its faces. A \emph{piecewise linear $d$-sphere} is a complex piecewise linearly homeomorphic to $\partial \Delta^{d+1}$. For $d\geq 5$, there are simplicial complexes whose underlying space is homeomorphic to the $d$-sphere, but which are not piecewise linear spheres.
Any PL sphere is actually a subdivision of $\partial \Delta^{d+1}$.

The preceding definitions refer to linear embeddings of simplicial complexes. The notion of piecewise linear homeomorphism, and hence piecewise linear balls and spheres, can also be defined purely combinatorially in terms of \emph{stellar} and \emph{bistellar moves}. \emph{Stellar subdivision} at a (nonempty) face $\sigma$ of a simplical complex $\K$ is the operation that removes the star $\st(\sigma,\K)$ and replaces it by $v\ast \partial  \sigma \ast \lk(\sigma, \K)$, where $v$ is a new vertex (geometrically, we place $v$ into the relative interior of $\sigma$ and subdivide $\sigma$ and all faces containing it by coning from $v$). The inverse operation is called a \emph{stellar weld}.

Any subdivision of a finite simplicial complex can be obtained by a finite sequence of stellar moves (subdivisions and welds). Consequently, two finite simplicial complexes are piecewise linearly isomorphic iff one can be obtained from the other by a finite sequence of stellar moves.

The combinatorial changes during a stellar subdivision depend on the link of the face $\sigma$. Thus, even if the dimension of the complex is fixed, there are infinitely many different such moves.

Pachner~\cite{Pachner:Konstruktionsmethoden-1987} showed that for piecewise linear balls and spheres
(and more generally, for piecewise linear closed manifolds), for each given dimension, there is a finite list of $d+1$ operations that suffice
for moving from one triangulated space to any PL-homeomorphic triangulation of it.
Suppose that $\sigma $ is a $p$-dimensional simplex in a
simplicial complex $\K$ whose link $\lk(\sigma, \K)$ is isomorphic
to $\partial \tau$ for some $q$-dimensional simplex $\tau \not \in
\K$ (i.e., $\tau$ is a missing face of $\K$). The operation of
removing $\st(\sigma)=\sigma\ast \partial \tau$ and replacing it
by $\partial \sigma \ast \tau$ is called a \emph{bistellar move}
of type $(p,q)$, also known as Pachner move in the literature.
\begin{theorem}[Pachner~\cite{Pachner:Konstruktionsmethoden-1987}]\
\label{fact:PL-bistellar}
A finite simplicial complex is a piecewise linear $d$-sphere \textup{[}resp.\ $d$-ball\textup{]} iff it can be obtained from $\partial \Delta^{d+1}$ \textup{[}resp.\ $\Delta^d$\textup{]} by a finite sequence of bistellar moves \textup{(}each of them of some type $(p,d-p)$, $0\leq p \leq d$\textup{)}.
\end{theorem}

We will need another fact about piecewise linear spheres (which fails in the non-PL case; for example the Alexander ``horned 2-sphere'' in $\R^3$ has the property that the unbounded component is not simply connected).
\begin{theorem}[Newman~\cite{Newman:FoundationsCombinatorialAnalysisSitus-1926}]
\label{Fact:Newman}
Let $\complex{S}$ be a piecewise linear $d$-sphere that contains a subcomplex $\complex{B}$ which is a piecewise linear $d$-ball. Then the closure $\overline{\complex{S}-\complex{B}}$ is a piecewise linear $d$-ball.
\end{theorem}

\paragraph{Deleted Products and the van Kampen obstruction.}

We recall the deleted product construction and the definition of the $\Z_2$-valued van Kampen obstruction.

The (twofold) \emph{deleted product} of $\K$ is defined as $$\delprod{\K}=\{\sigma \times \tau: \sigma,\tau \in \K, \sigma \cap \tau=\emptyset\}.$$ Thus, the cells of the deleted product are cartesian products of vertex-disjoint simplices of $\K$. The deleted product comes equipped with an obvious cellular action of $\Z_2$ that simply exchanges the order of factors, $\sigma \times \tau \mapsto \tau\times \sigma$; the action is \emph{free}, i.e., it does not have any fixed-points. We remark that up to $\Z_2$-equivariant homotopy equivalence, the deleted product only depends on the underlying space $\|\K\|$, not on the particular triangulation. (A \emph{$\Z_2$-equivariant map}, or simply \emph{$\Z_2$-map}, between two spaces with $\Z_2$-actions is one that commutes with the respective $\Z_2$-actions.)

It follows from the general theory of principal bundles (see, for instance, \cite{Husemoller:FibreBundles-1994} or \cite{TomDieck:TransformationGroups-1987}) that there
is a $\Z_2$-map $\delprod{\K}\rightarrow \s^\infty$, which is unique up to $\Z_2$-homotopy. This induces a unique map (up to homotopy) between the quotient spaces $\delprod{\K}/\Z_2\rightarrow \R P^\infty$, and hence there is a unique homomorphism (up to isomorphism) in cohomology $H^\ast (\R P^\infty; \Z_2)\rightarrow H^\ast (\delprod{\K}/\Z_2 ; \Z_2)$. The $\Z_2$-cohomology ring of $\R P^\infty$ is isomorphic to the polynomial ring $\Z_2[x]$, and the image of the generator $x$ under
the above homomorphism is called the first \emph{Stiefel-Whitney
class} of $\delprod{K}/\Z_2$, denoted by $\varpi_1(\delprod{K}/\Z^2)$. Like the deleted product, it only depends on the underlying space $\|\K\|$.

The $r^\textrm{th}$ power $\varpi_1^r(\delprod{\K}/\Z_2) \in
H^r(\delprod{\K}/\Z_2; \Z_2)$ is called the \emph{van Kampen
obstruction} (with coefficients in $\Z_2$) for embeddability into
$\R^r$ and denoted by $\obsr(\K)$.
If $\obsr(\K)\neq 0$ then $\|\K\|$ does not embed in $\R^r$:
if $f:\|\K\| \rightarrow \R^r$ is an embedding then $(x,y)\mapsto \frac{f(x)-f(y)}{\|f(x)-f(y)\|}$ defines
a $\Z_2$-equivariant map
 $\|\delprod{\K}\| \rightarrow \s^{r-1} \subset \s^\infty$.
 Thus, on the level of cohomology,
the map $H^\ast (\R P^\infty; \Z_2)\rightarrow H^\ast (\delprod{\K}/\Z_2 ; \Z_2)$ factors through
$H^\ast (\R P^{r-1}; \Z_2)\cong \Z_2[x]/(x^r)$, so the image $\obsr(\K)$ of $x^r$ is zero.
We are mainly interested in the case $r=2d$, where $d=\dim \K$ (recall that $\K$ always linearly embeds in
$\R^{2d+1}$, by placing the vertices in general position, for instance on the moment curve).
The $\Z_2$-valued van Kampen obstruction is \emph{complete}, i.e., $\obstwod(\K)=0$ iff $\K$ embeds into $\R^{2d}$, for $d=1$ (as was shown by Hanani~\cite{Hanani:UnplattbareKurven-1934} and independently by Tutte \cite{Tutte:TowardATheoryOfCrossingNumbers-1970}, see also \cite{Sarkaria:OnedimensionalWhitneyTrick-1991} for a cohomological treatment). For $d>1$; the $\Z_2$-valued obstruction is incomplete (see \cite[Example~2.4]{Melikhov:vanKampenRelatives-2006}). We remark that there is an integer-valued variant of the van Kampen obstruction, which is a complete obstruction for $d\neq 2$ (see, e.g., \cite{Shapiro:FirstObstruction-1957}), but still incomplete for $d=2$
(see \cite{FreedmanKrushkalTeichner:vanKampenObstructionIncomplete-1994}).

\paragraph{Linking and intersection numbers.}

To compute van Kampen's obstruction, we will mostly work with a particular
representative of the cohomology class, namely the intersection cochain induced
by a map in general position (this is essentially van Kampen's original definition,
which predates the formal definition of cohomology; see
Shapiro~\cite{Shapiro:FirstObstruction-1957} for a detailed presentation with full proofs).

Let $\K$ be a finite simplicial complex
and let $f:\|\K\|\rightarrow \R^{r}$ be a continuous map \emph{in general position}, i.e., $f(\sigma^p)\cap f(\tau^q)=\emptyset$ whenever $p+q<r$ (where we write $\sigma^p$ for $\sigma \in \K$ with $\dim \sigma=p$).
Then one can define an \emph{intersection number} $f(\sigma^p)\cdot f(\tau^q)\in \Z_2$ for any pair of (oriented) simplices with $p+q=r$. We refer to Dold \cite[Chapter VII, \S 4]{Dold:LecturesAlgebraicTopology-1995} for a formal homological definition of intersection numbers; geometrically speaking, up to a small perturbation (which preserves general position), one can assume that $f$ is a smooth (or PL) map and that any two simplices $\sigma^p$ and $\tau^q$ with $p+q=r$
have only a finite number of intersection points each of which is transverse (i.e., the tangent spaces of $f(\sigma^p)$ and $f(\tau^q)$ at the intersection point are complementary subspaces of $\R^r$); then
$f(\sigma^p)\cdot f(\tau^q)$ is simply the number of intersection points modulo 2. (For integer-valued
intersection numbers of images of oriented simplices, one counts the intersections with signs, depending on the orientation of $\R^r$ induced by the orientations of the tangent spaces.)

Clearly, the $\Z_2$-valued intersection number is commutative; moreover, it has the property that
\begin{equation}
f(\partial \sigma^{p}) \cdot f(\tau^q)= f(\sigma^{p})\cdot f(\partial \tau^q), \tag{$p+q=r+1$}
\end{equation}
with the understanding that the intersection number is extended from simplices to chains by bilinearity (for $\Z$-valued intersection numbers, one has correction signs $(-1)^{pq}$ when exchanging factors, and $(-1)^p$ when applying the boundary operator on different sides). This property allows one to define the homological \emph{linking number} of (the images of) boundaries $\partial \alpha^p$ and $\partial \beta^q$ of vertex -disjoint chains with $p+q=r+1$ as the intersection number $f(\partial \alpha) \cdot f(\beta)$.

In the case $r=2d$, $d=\dim \K$, any map $f$ as above defines an \emph{intersection number cochain}
$\varphi_f \in C^{2d} (\delprod{\K}/\Z_2; \Z_2)$ by
$$\varphi( \{\sigma,\tau\}):=f(\sigma)\cdot f(\tau)=|f(\sigma)\cap f(\tau)|\imod 2.$$
Here, we identify the cells in the quotient space $\delprod{\K}/\Z_2$
with unordered pairs $\{\sigma,\tau\}$ of vertex disjoint simplices
$\sigma,\tau\in \K$, $\dim\{\sigma,\tau\}=\dim\sigma+\dim\tau$.

Since we are working with coefficients in $\Z_2$, there is an obvious bijection between $r$-chains and sets of $r$-cells, and we will freely switch back and forth between these two viewpoints.
If $\alpha$ is a chain, we write
$\supp(\alpha)=\bigcup_{\sigma \in \alpha}\sigma$ for the
\emph{support} of $\alpha$, i.e., the  union of the vertices of
simplices in $\alpha$. Moreover, we extend the unordered pair notation to chains by linearity, i.e.,
if $\alpha,\beta$ are chains in $\K$ with disjoint supports, we use $\{\alpha,\beta\}$ to denote the
chain $\sum_{\sigma \in \alpha,\tau\in \beta} \{\sigma,\tau\}$.
In particular, by definition of the boundary of cartesian products,
$$\partial \{\sigma,\tau\} =\{\partial \sigma,\tau\} +\{\sigma, \partial \tau\}.$$
It follows that a chain $\omega \in C_r(\delprod{\K}/\Z_2; \Z_2)$ is a cycle iff
for every $p$-simplex $\sigma \in \K$, the $(r-p)$-chain
$\{\tau \in \K: \{\sigma,\tau\}\in \omega\}$ is a cycle.

It turns out that the intersection number cochain $\varphi_f$ is a cocycle, and that the cohomology class $[\varphi_f] \in H^{2d}(\delprod{\K}/\Z_2; \Z_2)$ determined by this cocycle
is independent of $f$; more specifically, $[\varphi_f]=\obstwod(\K)$.

Since we consider (co)homology over a field, there is no torsion and
the cohomology class $\obstwod(\K)$ does not vanish iff there is a
\emph{witness} homology class $\omega \in H_{2d} (\delprod{\K}/\Z_2;
\Z_2)$ such that $\langle \obstwod(\K),\omega)\rangle \neq 0$, where
$\langle,\rangle$ is the natural pairing between cohomology and
homology. Thus, using intersection number representatives, it follows
that $\obstwod(\K)\neq 0$ iff there is a $2d$-cycle $\omega \in C_{2d}(\delprod{\K}/\Z_2;
\Z_2)$ such that
$$\sum_{\{\sigma,\tau\}\in \omega} | f(\sigma)\cap f(\tau)|\neq 0 \imod 2.$$

To prove Theorem \ref{thmVK-F}, take $\omega=\sum \{\{\sigma,\tau\}: V\supseteq \sigma,\tau\neq \emptyset,\ \sigma\cap \tau=\emptyset, |\sigma|=|\tau|=d+1 \}$ where $V$ is the vertex set of the $(2d+2)$-simplex, and $f$ is obtained by
locating the vertices on different points on the moment curve in $\R^{2d}$ and
and taking convex hulls for the relevant simplices.

\paragraph{Further notation.} We also extend the join notation bilinearly to chains with disjoint supports, i.e.,
$\alpha\ast \beta =\sum_{\sigma\in \alpha,\tau\in \beta} \sigma\ast
\tau$ (recall that for abstract, vertex-disjoint simplices, $\sigma \ast \tau =\sigma \cup \tau$; geometrically, we embed the two simplices linearly into skew subspaces of some euclidean space and take the convex hull of their union). Finally, for simplices $\rho, \tau$, we write $\rho \prec \tau$ if
$\rho$ is a facet of $\tau$.

\section{Warm-up: Proof of Theorem \ref{thmHighSkel}}\label{sec:HighSkel}
Theorem \ref{thmHighSkel} follows from the following theorem, inspired by
Lo\'vasz and Schrijver \cite[Theorem 1]{LovaszSchrijver:BorsukTheoremAntipodalLinks-1998},
Theorem 1, which deals with antipodal faces in convex polytopes.

\begin{proposition}[Borsuk-Ulam Theorem]\label{thmOddIntersection}
Let $\K$ be a triangulated $d$-sphere and let $f: \K\longrightarrow
\mathbf{R}^d$ be a continuous map in general position. Then there exist two
disjoint faces $\sigma, \tau\in \K$ such that $\dim \sigma +\dim \tau=d$
and $|f(\sigma)\cap f(\tau)|$ is odd.
\end{proposition}

\begin{proof}
First assume that $\K=\partial\Delta^{d+1}$, where $\Delta^i$ is the $i$-simplex. Let $f:\K\rightarrow \R^d$ be the map that produces the Schlegel diagram of $\Delta^{d+1}$ through an (arbitrary, but fixed) facet $F$, and the consider the $d$-cycle $\omega=\sum\{ \{\sigma,\tau\} : \sigma,\tau \neq \emptyset, \sigma \cap \tau=\emptyset, \sigma \ast \tau=\Delta_0^{d+1}\}$. We have $\obsd(\K)(\omega)=|f(F)\cap f(v)|=1$ where $v$ is the unique vertex in $\K$ not in $F$. Thus $\obsd(\K)\neq 0$. Since the obstruction is independent of the particular triangulation and of the map $f$, it follows that for any $\K$ and $f$ as in the theorem, there exists a $2d$-cycle $\omega$ in $C_{2d}(\delprod{\K}/\Z_2;
\Z_2)$ such that $\sum_{\{\sigma,\tau\}\in \omega}|f(\sigma)\cap f(\tau)|\neq 0 \imod{2}$. In particular, this sum contains an odd number of nonzero summands, any one of which yields a pair of faces $\sigma,\tau$ as in the conclusion of the theorem.
\end{proof}

We now prove Theorem \ref{thmHighSkel}, following the reasoning in \cite{LovaszSchrijver:BorsukTheoremAntipodalLinks-1998}. Let $\K$ be the $d$-skeleton of a triangulated $(2d+1)$-sphere. First observe that the theorem is obviously true for $d=0$. For $d>0$, embeddability of the $d$-dimensional complex $\K$ into $\s^{2d}$ is equivalent to embeddability into $\R^{2d}$. Thus, assume for a contradiction that there exists an embedding $f:\|K\|\hookrightarrow \R^{2d}$. We identify $\R^{2d}\cong \R^{2d}\times \{0\}\subseteq \R^{2d+1}$ and extend $f$ to a map $g:\complex{S}\rightarrow \R^{2d+1}$, where $\complex{S}$ is any triangulation of the $(2d+1)$-sphere with $\complex{S}_{\leq d}=\K$; the only property we will need is that we can choose this extension $g$ such that $g(\|\complex{S}\|\setminus \|\K\|)$ is contained in the open upper halfspace $\R^{2d}\times \R_{>0}$. Then, for any two disjoint faces $
\sigma,\tau \in \complex{S}$ whose dimensions add up to $2d+1$, we have $g(\sigma)\cap g(\tau)=\emptyset$, since one of the faces is of dimension $\leq d$ and the other is of dimension $\geq d+1$. This contradicts Proposition~\ref{thmOddIntersection}. $\square$

\section{Proof of Theorem \ref{thmNonSkeleton}}\label{sec:Dancis}
This theorem is inspired by the proof of Dancis \cite{Dancis:ComplexesDeterminedBySkeleta-1984} that a triangulated compact $2d$-manifold without boundary and with vanishing $d$-homology is determined by its $d$-skeleton. (The first result of this kind goes back to Perles, who considered the polytope boundary case.)

Dancis' proof is based on constructing inductively the $k$-skeleton from the $(k-1)$-skeleton according to homological criteria which tells us which $(k-1)$-dimensional subcomplexes isomorphic to $\partial \Delta^k$ we should fill with a $k$-simplex \cite[proofs of Theorems 11 and 13]{Dancis:ComplexesDeterminedBySkeleta-1984}. More specifically, for a triangulated $2d$-sphere $\complex{S}$, assume that $\complex{S}_{< k}$ is given, and that $\sigma$ is $k$-dimensional simplex with $\partial \sigma \subseteq \complex{S}_{< k}$. If $k=d+1$ then $\sigma \in \complex{S}_k$
iff $\tilde{H}_{d-1}(C(\partial \sigma,\complex{S})_{\leq d};\Z)=0$ where $C(\complex{T},\complex{S})$ is the induced subcomplex of $\complex{S}$ on the vertices not in the subcomplex $\complex{T}$. If $k>d+1$, then $\sigma \in \complex{S}_k$ iff $\tilde{H}_{i}(C(\partial \sigma,\complex{S})_{\leq k-1};\Z)=0$ for $i=2d-k+1,2d-k$.

\begin{proof}[Proof of Theorem~\ref{thmNonSkeleton}] Assume by contradiction that $\complex{S}$, $\complex{B}$, and $g$ satisfy the assumptions of the theorem and that $\K:=\complex{S}_d\cup_g \complex{B}$ is the $d$-skeleton of the triangulated $2d$-sphere $\complex{S}'$ (by Dancis result, $\complex{S}'$ is determined by $\K$). For $i<d$ we have $\complex{S}_i=\complex{S}'_i$. Moreover, we have $\complex{S}_d \subset \complex{S}'_d$. Thus, whenever the homology groups in the above criteria vanish for  $\complex{S}$, they also vanish for $\complex{S}'$. Therefore, we get that $\complex{S}$ is a proper subcomplex of $\complex{S}'$ and at the same time homeomorphic to it. This is a contradiction, for instance by Alexander-Poincar\'e duality. \end{proof}

\section{Van Kampen's Obstruction and Bistellar Moves}\label{sec:bistellar}

We consider the behaviour of the van Kampen obstruction under
bistellar moves. Let $\K$ be a $d$-dimensional simplicial complex,
and let $\complex{T}$ be an \emph{induced} subcomplex
of $\K$ that is isomorphic to
the $d$-skeleton of $\Delta^p \ast \partial \Delta^q$, for some
$p+q=2d$, and such that the open star $\{\sigma\in \K: \Delta^p\subseteq
\sigma\}\subseteq \complex{T}$. Let $\complex{L}$ be the $d$-dimensional complex obtained from
$\K$ by replacing $\complex{T}$ by the $d$-skeleton of
$\partial\Delta^p \ast \Delta^q$, which we denote by $\complex{T}'$.
Note that $\complex{T}\cap \complex{T}'$ is the $d$-skeleton of
$\partial \Delta^p \ast \partial \Delta^q$ (in particular, if $p>0$
then $\complex{T}$ and $\complex{T}'$ have the same vertex set). In
other words, we delete all $d$-faces of the form $\Delta^p\ast
\rho$, $\rho \in \partial \Delta^q$, from $\K$ and add all $d$-faces
of the form $\sigma\ast \Delta^q$, $\sigma \in
\partial\Delta^p$; here we abuse notation and use the same
symbol $\Delta^p$ to denote both a fixed $p$-dimensional simplex as
well as the simplicial complex consisting of that simplex and all
its faces, depending on the context. We say that $\complex{L}$
arises from $\K$ by a \emph{bistellar move of type $(p,q)$}. Note
that, in particular, if $p>d$, then the old complex, $\K$, is a
subcomplex of the new one, $\complex{L}$.

\paragraph{The assumptions.} We wish to show that under certain assumptions, a nonvanishing van Kampen obstruction is preserved by a bistellar move as above.
We will  work with the intersection cochain
representative $\varphi_f$ described in Section \ref{sec:Background}.
Fix a continuous map $f:\|\K\|\rightarrow \R^{2d}$ in general position.
Our first assumption is that all intersections happen ``outside of
$\complex{T}$'', i.e.,
\begin{equation}
\label{eq:assumption-1a}
\sigma \in \complex{T} \textrm{ or } \tau\in \complex{T} \quad \Longrightarrow \quad \varphi_f(\{\sigma,\tau\})=0.\tag{I.a}
\end{equation}
Similarly, we assume that there is a continuous map $g:\|\complex{L}\|\rightarrow
\R^{2d}$ in general position that agrees with $f$ on $\|\K\cap \complex{L}\|$ and such
that the corresponding intersection cochain $\varphi_g$ satisfies
\begin{equation}
\label{eq:assumption-1b}
\sigma \in \complex{T}' \textrm{ or } \tau\in \complex{T}' \quad \Longrightarrow \quad \varphi_g(\{\sigma,\tau\})=0.\tag{I.b}
\end{equation}

Our next assumption is that the homology of $\K$ ``outside of
$\complex{T}$'' is sufficiently simple. More precisely, we assume
that
\begin{equation}
\label{eq:assumption-2a} \exists v_0 \in \K_0 \setminus
\complex{T}_0=\complex{L}_0 \setminus \complex{T}'_0 \tag{II.a}
\end{equation}
and that
\begin{equation}
\label{eq:assumption-2b}
\widetilde{H}_i(\K-W;\Z_2)=0, \qquad 0\leq i \leq d-1
%\widetilde{H}_i(\K\cap\complex{L} -W;\Z_2)=0, \qquad 0\leq i \leq d-1
\tag{II.b}
\end{equation}
whenever $W=\complex{T}_0 \setminus \rho$ for some face $\rho \in \complex{T} \cap \complex{T}'$ (where $\K-W$ is obtained from $\K$ by deleting all faces containing at least one vertex in $W$).
Observe that if $W':=\complex{T}'_0 \setminus \rho$ and $W'':=(\complex{T}_0\cap \complex{T}'_0) \setminus \rho$ then $\complex{L}-W' = \K-W = \K \cap \complex{L}-W''$.

\begin{observation}
\label{obs:PLspheres-are-good} Suppose that $\K=\complex{S}_{\leq d}\cup\{M\}$ is the $d$-skeleton of a PL $2d$-sphere $\complex{S}$ plus a missing $d$-face $M$ and that $\complex{S}'$ is the $2d$-sphere obtained by a bistellar move of type $(p,q)$, i.e., by replacing a $2d$-ball $\complex{B}=
\Delta^p \ast \partial \Delta^q \subset \complex{S}$ by the $2d$-ball $\complex{B}'=\partial \Delta^p \ast \Delta^q$. Suppose furthermore that $M$ is also a missing $d$-face of $\complex{S}'$.
If we set $\complex{L}:=\complex{S}'_{\leq d}\cup\{M\}$, $\complex{T}:=\complex{B}_{\leq d}$, and $\complex{T'}:=\complex{B}'_{\leq d}$, then all assumptions \textup{(\ref{eq:assumption-1a})--(\ref{eq:assumption-2b})}
are satisfied.

\end{observation}
\begin{proof}
Choose an arbitrary point $p\in \topint\|\complex{B}\|\setminus\|\complex{B}_d\|$ and fix a homeomorphism $\|\complex{S}\|\setminus\{p\} \cong
\R^{2d}$. Since we assume that $M$ remains a missing face after the bistellar move,
it follows that $\partial M \subseteq \K\cap \complex{L}$; in particular, $\|\partial M\| \cap \topint \|B\|=\emptyset$. By Newman's Theorem~\ref{Fact:Newman}, $\|\complex{S}\|\setminus \topint \|B\|$ is a PL $2d$-ball; in particular, it is contractible. Thus, we can extend the inclusion $\| \partial M \| \subset \| \complex{S}\|\setminus \topint \|B\| \subseteq \R^{2d}$ to an embedding of the disk $\|M\|\hookrightarrow \| \complex{S}\|\setminus \topint \|B\| \subseteq \R^{2d}$; moreover, by a suitable small perturbation, if necessary, we can choose this embedding so that it avoids $\|\complex{S}_{<d}\|$. Defining $f$ as this embedding on $M$ and as the restriction of the inclusion $\| \complex{S}_{\leq d} \| \subseteq \|\complex{S} \| \setminus \{p\} \cong \R^{2d}$ on $\|\complex{S}_{\leq d}\|$, we obtain a map $f:\|K\|\rightarrow \R^{2d}$ that satisfies (\ref{eq:assumption-1a}). Moreover, by choosing a homeomorphism $h:\|\complex{B}'\| \cong \|\complex{B}\|$ that extends the identity map on the boundary $\partial \complex{B}=\partial\complex{B}'$ we can define $g:=f \circ h |_{\|\complex{T}'\|}$ on $\complex{T}'=\complex{B}'_{\leq d}$
and $g=f$ on $\|\complex{K}\cap \complex{L}\|$. Thus, we obtain a map $g$ that satisfies (\ref{eq:assumption-1b}).

Since $\complex{B}$ is an induced subcomplex of $\complex{S}$ and is not a sphere, $\complex{B}_0$ is a \emph{proper} subset of $\complex{S}_0$, i.e.
there exists at least one vertex $v_0 \in \complex{S}_0 \setminus \complex{B}_0=\complex{S}_0' \setminus \complex{B}_0'$; hence (\ref{eq:assumption-2a}) holds.

Finally, to see that (\ref{eq:assumption-2b}) holds, let $\rho \in \complex{T}\cap \complex{T}' \subseteq (\partial \complex{B})_{\leq d}$
%=(\partial \complex{B}')_{\leq d}
and $W=\complex{T}_0 \setminus \rho$. The induced subcomplex $\complex{S}[W]=\complex{B}[W]$ is a cone, and therefore contractible: we have $\rho =\sigma \ast \tau \in \partial \Delta^p \ast \partial \Delta^q$; in particular, $\Delta^p - \sigma \neq \emptyset$, hence $\complex{B}[W]=(\Delta^p- \sigma) \ast (\partial \Delta^q - \tau)$ is indeed a cone (observe that the assumption $\rho \in \partial \complex{B}$ is essential). It follows from this and Alexander duality
\cite[Chapter 8, \S 71]{Munkres:AlgebraicTopology-1984} that $\|\complex{S}\| \setminus \|\complex{S}[W]\|$ has zero homology in all dimensions. Moreover, $\|\complex{S}-W\|$ is a retract of $\|\complex{S}\|\setminus \|\complex{S}[W]\|$, so they have the same homology.
Finally, $(\K-W)_{<d} = (\complex{S}-W)_{<d}$ and $(\K-W)_d \supseteq (\complex{S}-W)_{d}$, hence $\tilde{H}_i(\K-W; \Z_2) = \tilde{H}_i(\complex{S}-W; \Z_2)=0$ for $0\leq i \leq d-1$.
\end{proof}

We will use the homological assumptions in form of the following lemma:
\begin{lemma}
\label{lem:homological-coning} Assume that
\textup{(\ref{eq:assumption-2a})} and
\textup{(\ref{eq:assumption-2b})} hold. Then there is a
``homological coning'' operation that associates with every
$i$-simplex $\rho \in \K\cap \complex{L}$, $0\leq i \leq d-1$, an
$(i+1)$-chain $v_0 \bullet \rho\in C_{i+1}(\K\cap\complex{L};\Z_2)$
with the following properties:
\begin{enumerate}
\item If $\vartheta \in C_i (\K\cap \complex{L}; \Z_2)$ is an $i$-chain then the $(i+1)$-chain
$v_0 \bullet \vartheta:=\sum_{\rho \in \vartheta}v_0\bullet \rho$ is
as disjoint from $\complex{T}$ as possible, in the sense that any
vertex $w \in \complex{T}_0$ that lies in the support of $v_0
\bullet \vartheta$ already lies in the support of $\vartheta$.
\item If $\vartheta$ is an $i$-cycle, then $\partial (v_0\bullet
\vartheta)=\vartheta$.
\end{enumerate}
\end{lemma}
\begin{proof} If $v_0\in \rho$ define $v_0 \bullet \rho=0$. Otherwise, if $v_0
\ast \rho \in \K\cap \complex{L}$ define $v_0 \bullet
\rho=v_0 \ast \rho$. Note that in this case
\begin{equation}
\label{eq:coning}
\partial(v_0 \bullet \rho)=\rho+\sum_{u\in
\rho}v_0 \bullet (\rho\setminus \{u\}).
\end{equation}
Else, i.e. if $v_0 \ast \rho \notin \K\cap \complex{L}$, we
proceed by induction on dimension of $\rho$ (we may assume that
(\ref{eq:coning}) holds for proper subsets of $\rho$): consider the
already defined chain $z(\rho):=\rho+\sum_{u\in \rho}v_0 \bullet
(\rho\setminus \{u\})$. This is a cycle:
$\partial(z(\rho))=2\sum_{u\in \rho}\rho\setminus \{u\} + \sum_{u\in
\rho}\sum_{w: u\neq w\in \rho}v_0 \bullet (\rho\setminus
\{u,w\})=0$. By (\ref{eq:assumption-2b}), there is an $(i+1)$-chain
$b(\rho)\in C_{i+1}(\K\cap\complex{L}-W'';\Z_2)$ with
$\partial(b(\rho))=z(\rho)$ where
$W''=(\complex{T}_0\cap \complex{T}'_0) \setminus \rho$. Define $v_0 \bullet \rho=b(\rho)$;
thus equation (\ref{eq:coning}) holds for $\rho$.

By construction, the first property holds. As for the second
property, let $\vartheta$ be an $i$-cycle. Then $\partial
(v_0\bullet \vartheta)= \partial(\sum\{v_0 \bullet \rho: v_0\notin
\rho\in \vartheta\})= \sum\{\rho: v_0\notin \rho\in
\vartheta\}+\sum_{v_0\notin \rho\in \vartheta}\sum_{u\in \rho} v_0
\bullet (\rho\setminus u)= \sum_{v_0\notin \rho\in \vartheta}\rho +
\sum_{\rho': v_0\uplus \rho'\in \vartheta}v_0\bullet
\rho'=\vartheta$, where the second-to-last equality follows from
$\vartheta$ being a cycle over $\Z_2$.
\end{proof}
%
%\noindent We will also need the following simple auxiliary lemma:
%
%\begin{lemma}
%\label{lem:cycle-basis-simplex}
%Let $\Delta$ be a $q$-simplex, and let $v$ be an arbitrary vertex of $\Delta$.
%Let $i<q$, and let $\sigma_1,\ldots,\sigma_m$ be the $i$-simplices in $\Delta$ that are disjoint from $v$. Then the boundaries $\partial(v\ast \sigma_i)=v\ast \partial \sigma_i$, $1\leq i\leq m$, form a basis for the space $Z_i(\Delta)$ of $i$-cycles in $\Delta$.
%\end{lemma}
%\eran{'2d' instead of 'd' in the following:}
%
\begin{theorem}
\label{thm:bistellar-moves} Suppose that $\complex{L}$ arises from
$\complex{K}$ by a bistellar move of type $(p,q)$ as above and that
all assumptions \textup{(\ref{eq:assumption-1a})} --
\textup{(\ref{eq:assumption-2b})} are satisfied. Then
$\obstwod(\complex{K})\neq 0$ implies $\obstwod(\complex{L})\neq 0$.
\end{theorem}

\begin{proof} Let $\omega\in H_{2d}(\delprod{\complex{K}}/\Z_2; \Z_2)$ be a homology class witnessing the nonvanishing of
$\obstwod(\complex{K})$. Choose a $2d$-cycle representing this
homology class; by abuse of notation, we denote this cycle also by
$\omega$; thus, $\varphi_f(\omega)=\sum_{\{\sigma,\tau\}\in \omega}
f(\sigma)\cdot f(\tau)\neq 0$. If $\omega$ does not contain any
simplices containing $\Delta^p$ (for instance, this is automatically
the case if $p>d$), it is also a cycle in $C_{2d}(\delprod{\complex{L}}/\Z_2; \Z_2)$,
and by (\ref{eq:assumption-1b}), also $\varphi_g(\omega)\neq 0$, and
we are done.
%\eran{See change in the sentence below.}
Otherwise, we obtain a chain $\widetilde{\omega}$ from $\omega$ by removing all the pairs $\{\sigma,\tau\}$
involving simplices containing $\Delta^p$, and replace them by pairs
of simplices in $\complex{L}$ in such a way as to ensure three
properties:
\begin{enumerate}
\item[(i)] for every new pair $\{\sigma,\tau\}$ of simplices that we add, $\sigma$ and $\tau$ are disjoint, i.e. $\widetilde{\omega}\in C_{2d}(\delprod{\complex{L}}/\Z_2; \Z_2)$;
\item[(ii)] the resulting chain $\widetilde{\omega}$ is a cycle, i.e., for each $d$-simplex $\sigma$, the $d$-chain $\{\tau: \{\sigma, \tau \}\in \widetilde{\omega}\}$ is a cycle;
\item[(iii)] all pairs $\{\sigma,\tau\}$ that we delete have the property that at least one of the two simplices belongs to the subcomplex $\complex{T}$, and all pairs that we add contain at least one simplex in $\complex{T}'$; consequently, by \textup{(\ref{eq:assumption-1a})} and \textup{(\ref{eq:assumption-1b})},
$\varphi_f(\omega)=\varphi_g (\widetilde{\omega})$.
\end{enumerate}

\noindent We now describe in detail how this is done. As remarked above, we may assume that $p \leq d$.

Note that if $p<d$, then $q=2d-p>d$, and we do not add any new $d$-faces. If $p=d$, then $q=d$, and $\complex{L}$ contains exactly one new $d$-face, namely $\Delta^q$ itself; we remark that this $d$-face will not be used in the construction of $\widetilde{\omega}$.

\paragraph{Step 1: Removing the Simplices Containing $\Delta^p$.} For each $d$-simplex $\tau \in \complex{K}$, the set
$\alpha_\tau:=\{\sigma \in \complex{K}:  \{\sigma,\tau\}\in
\omega\}$  of $d$-simplices paired up with $\tau$ in $\omega$ is a
$d$-cycle. In particular, the subset of $\sigma$'s containing
$\Delta^p$ is of the form $\{\sigma \in \complex{L}:
\{\sigma,\tau\}\in \omega,\Delta^p\subseteq \sigma\}=\Delta^p \ast
\gamma_\tau$ for a uniquely determined $(d-p-1)$-cycle $\gamma_\tau
\in Z_{d-p-1}(\partial \Delta^q)$ whose support is disjoint from
$\tau$
(namely, $\gamma_\tau$ is the link of $\Delta^p$ in $\alpha_\tau$; here
we use the assumption that $\st(\Delta^p,\K)\subseteq \complex{T}$).
To see that $\gamma_\tau$ is a cycle, note that $\partial
\gamma_\tau \neq 0$ implies $0\neq\Delta^p \ast \partial \gamma_\tau
\subseteq \partial \alpha_\tau$, a contradiction. Observe that in
the case $p=d$, the cycle $\gamma_\tau$ is either zero or consists
of the single $(-1)$-dimensional face $\emptyset$. Moreover, for any
$p$, $\gamma_\tau=0$ for all $\tau \in \complex{T}$. This is clear
if $\tau \cap \Delta^p \neq \emptyset$; otherwise, $\tau \in
\partial \Delta^q$, and the support of any nontrivial
$(d-p-1)$-cycle consists of at least
$d-p+1$
vertices
if $p<d$
, so there is not enough room in $\partial \Delta^q$ for both $\tau$
and a nontrivial $\gamma_\tau$ disjoint from it; if $p=d$, then
$q=d$, and $\partial\Delta^q$ does not contain any $d$-simplex
$\tau$. Let $\omega'$ be the chain obtained from $\omega$ by
deleting all pairs of simplices one of which contains $\Delta^p$,
$$\omega':=\omega - \sum\big\{ \{\sigma, \tau\} \in \omega: \Delta^p \subseteq \sigma\big\}.$$
Clearly, (iii) is satisfied. Observe, however, that $\omega'$
generally fails to be a cycle, since for each $\tau$, the chain
$\alpha'_\tau:=\{\sigma: \{\sigma,\tau\}\in \omega'\}$ has boundary
$\partial \Delta^p\ast \gamma_\tau$, where $\partial \Delta^p$
denotes the $(p-1)$-chain consisting of the facets of $\Delta^p$,
which is nonzero iff $\gamma_\tau\neq 0$.

\paragraph{Step 2: Patching holes  of the first kind.}
The idea is to repair this as follows: The complex $\Delta^q$
and all of its induced subcomplexes have vanishing homology. Recall
that $q=2d-p$.
 In particular, for each nontrivial $(d-p-1)$-cycle
$\gamma$ in $\partial\Delta^q$, we can choose a ``patching''
$(d-p)$-chain $\beta =\beta(\gamma)\in C_{d-p}(\Delta^q)$ with
$\partial \beta= \gamma$. If $p<d$, i.e., $d-p-1\geq 0$, then we can
choose $\beta(\gamma)$ to have the same support as $\gamma$, so that
$\beta_\tau:=\beta(\gamma_\tau)$ and $\tau$ are disjoint for every
$d$-simplex $\tau$. If $p=d=q$, we cannot choose one patching
$0$-chain $\beta$ that works for all $\tau$: any vertex $w\in
\Delta^q_0$ satisfies $\partial w=\emptyset$, but there may be
$d$-simplices $\tau$ with $\gamma_\tau=\emptyset$ yet $w\in \tau$.
However, since $\Delta^q\not \in \K$, this cannot be the case for
all $w$ for a given $\tau$.
Thus, if we choose an arbitrary ordering $w_0,w_1,\ldots,w_d$ of the
vertices of $\Delta^q$, we can define $\beta_\tau$ to be the first
$w_i \not \in \tau$, and again $\beta_\tau$ and $\tau$ are disjoint
for all $\tau$. We set
\begin{equation*}
\omega'':=\omega' + \big\{ \{\partial \Delta^p \ast
\beta_\tau,\tau\} :  \gamma_\tau\neq 0\big\}. %\tag{$p\leq d$}
\end{equation*}
Note that $\partial(\partial \Delta^p \ast
\beta(\gamma_\tau))=\partial \Delta^p \ast \gamma_\tau=\partial
\alpha'_\tau$. Consider a $d$-simplex $\sigma$ that is not of the
form $\sigma=\sigma^{p-1}\ast \sigma^{d-p}$ with $\sigma^{p-1}\in
\partial \Delta^p$ and $\sigma^{d-p} \subseteq \Delta^q$; we call
these simplices \emph{unproblematic}. Such an unproblematic simplex
does not appear in any of the patching chains $\partial \Delta^p
\ast \beta_\tau$, hence it is paired up in $\omega''$ with
$\alpha''_\sigma=\alpha'_\sigma + \partial\Delta^p\ast
\beta_\sigma$, which is a $d$-cycle by construction. On the other
hand, a \emph{problematic} simplex $\sigma =\sigma^{p-1}\ast
\sigma^{d-p}$ with $\sigma^{p-1}\in \partial \Delta^p$ and
$\sigma^{d-p}\subseteq \Delta^q$ is paired up in $\omega''$ with the
$d$-chain
$$\alpha''_\sigma =\alpha_\sigma + \{\tau : \sigma^{d-p} \in \beta_\tau\},$$
which need not be a cycle. We will take care of these ``holes of the second kind'' in Step~3 below. Observe, that the other two requirements (i) and (iii) are still met: By the way we choose the $\beta_\tau$, all added pairs have the property that the simplices are disjoint, and each patching chain $\partial\Delta^p\ast \beta$ is contained in $\complex{T}'$.

\paragraph{Step 3: Patching holes of the second kind.}
As remarked above, if $\sigma=\sigma^{p-1} \ast \sigma^{d-p}$ is a
problematic $d$-simplex in $\complex{L}$ (i.e., $\sigma^{p-1} \in
\partial \Delta^p$ and $\sigma^{d-p}$ a face of $\Delta^q$) then
$\sigma$ is paired up in $\omega''$ with the chain
$\alpha''_\sigma$, which need not be a cycle. More precisely, a
$(d-1)$-simplex $\rho \in \complex{L}$ appears on the boundary
$\partial \alpha''_\sigma = \sum\{ \partial \tau : \sigma^{d-p} \in
\beta_{\tau}\}$ with coefficient
$$|\{\tau: \  \tau \succ \rho,\  \sigma^{d-p} \in \beta_\tau\}| \pmod{2},$$
i.e. $\rho \in \partial\alpha''_\sigma$ iff there are an odd number of $d$-simplices  $\tau \succ \rho$ with $\sigma^{d-p} \in \beta_\tau$.
In other words, if we consider the $(d-p)$-chain in $\Delta^q$ given
by
$$\zeta(\rho):=\sum_{\tau \succ \rho} \beta_\tau,$$
then $\rho \in \partial \alpha''_\sigma$ iff $\sigma^{d-p} \in
\zeta(\rho)$; in particular, $\partial \alpha''_\sigma$ depends only
on $\sigma^{d-p}$,
and not on  $\sigma^{p-1}$.
%
%\eran{The last sentence is the crucial point that I (we) missed
%before, and enables your simplified proof, which avoids bases.}

\begin{claim} For each $(d-1)$-simplex $\rho$, the $(d-p)$-chain $\zeta(\rho)$ is a cycle.

\end{claim}
\begin{proof}[Proof of Claim 1]
For each $\rho$, consider the $(d-p-1)$-chain $\sum_{\tau\succ \rho} \gamma_\tau$ in $\Delta^q$. By definition of $\gamma_\tau$, a $(d-p-1)$-simplex $\pi$ belongs to this chain iff there are an odd number of $\tau \succ \rho$ with $\pi \in \gamma_\tau$, i.e., with $\tau \in \alpha_{\Delta^p \ast \pi}$, which is the case iff $\rho \in \partial \alpha_{\Delta^p \ast \pi}$. However, $\alpha_\sigma$ is a $d$-cycle for all $d$-simplices $\sigma \in \complex{K}$. Therefore, the $(d-p-1)$-chain $\sum_{\tau\succ \rho} \gamma_\tau$ is zero, and consequently the $d$-chain $\zeta(\rho)=\sum_{\tau\succ \rho} \beta_\tau$ is a cycle.
\end{proof}

Define
$$\widetilde{\omega}:=\omega'' + \big\{ \{\sigma, v_0 \bullet \partial \alpha''_\sigma\}: \sigma \textrm{ problematic}\big\}.$$
Every newly added pair $\{\sigma,\tau\}$ of simplices contains one
problematic simplex $\sigma \in \complex{T}'$, and one
nonproblematic simplex $\tau$. This immediately shows (iii).
Moreover, every nonproblematic simplex $\tau$ is paired up with
%
%\eran{Here I added some details:}
%%\uli{I apologize, but I removed the details you added again, mostly because I object to the notation $\sum_\sigma$ --- after all, we may be summing over some $\sigma$'s, but the result depends on $\tau$, and I also do not like the indicator function notation, and I am too tired to think of a better notation.}
%%
%$$\alpha''_\tau + \Sigma_\sigma$$
%where $\Sigma_\sigma := \{\sigma: \sigma \textrm{ problematic},
%\exists\textrm{ odd number of } \rho \in \partial \alpha''_\sigma
%\textrm{ with } \tau \in v_0\bullet \rho\}$. Summing $\rm{mod}\ 2$,
%we compute $\Sigma_\sigma = \sum_{\sigma \textrm{
%problematic}}\sigma \sum_{\rho:\ \tau \in v_0\bullet \rho}1_{\rho
%\in \partial \alpha''_\sigma} = \sum_{\rho:\ \tau \in v_0\bullet
%\rho}\sum_{\sigma \textrm{ problematic}}\sigma 1_{\sigma^{d-p}\in
%\zeta(\rho)} = \sum\{ \partial \Delta^p \ast \zeta(\rho):
%\rho\subseteq \complex{S'}_{d-1},\ \tau \in v_0\bullet \rho\}$.
%($1_a$ is the indicator function of the event $a$.)
%Thus $\alpha''_\tau + \Sigma_\sigma$ is a sum of cycles by
%construction ($\alpha''_\tau$) and by Claim~1 ($\Sigma_\sigma$).
%
$$\alpha''_\tau + \underbrace{\{\sigma: \sigma \textrm{ problematic}, \exists\textrm{ odd number of } \rho \in \partial \alpha''_\sigma \textrm{ with } \tau \in v_0\bullet \rho\}}_{\textstyle \sum\big\{ \partial \Delta^p \ast \zeta(\rho): \tau \in v_0\bullet \rho \big\}},$$
which is a sum of cycles by construction and by Claim~1.

On the other hand, each problematic simplex $\sigma$ is paired up with
$$\alpha''_\sigma + v_0\bullet \partial\alpha''_\sigma,$$
which is a cycle and disjoint from $\sigma$, by
Lemma~\ref{lem:homological-coning}. This shows both (i) and (ii) and
completes the proof of the theorem.
\end{proof}

%\begin{theorem}[Theorem \ref{thmMissingFacePL}]
%Let $\complex{S}$ be a PL $2d$-sphere plus a missing face. Then the following holds:
%\begin{itemize}
%\item[$(\ast)$] For every missing $d$-face $M$ of $\K$, $\obstwod(\K\cup\{M\})\neq 0$.
%\end{itemize}
%In particular, $\K\cup\{M\}$ does not embed in the $\R^{2d}$.
%\end{theorem}
\begin{proof}[Proof of Theorem~\ref{thmMissingFacePL}]
Let $\complex{S}$ be a PL $2d$-sphere. We want to show the following
\begin{itemize}
\item[$(\ast)$] For every missing $d$-face $M$ of $\complex{S}$, the complex $\K:=\complex{S}_{\leq d}\cup\{M\}$ satisfies $\obstwod(\K)\neq 0$.
\end{itemize}
%In particular, $\K$ does not embed into $\R^{2d}$.

By Pachner's result (Theorem~\ref{fact:PL-bistellar}), any PL $2d$-sphere can be obtained from the boundary of the $(2d+1)$-simplex by a finite sequence of bistellar moves. We prove $(\ast)$ by induction on the number of bistellar moves. The base case $\complex{S}=\partial \Delta^{2d+1}$ is trivial, since there are no missing faces at all.

For the induction step, let $\complex{S}$ be a PL $2d$-sphere that satisfies $(\ast)$, and let $\complex{S}'$ be obtained  from $\complex{S}$ by the bistellar move that replaces $\Delta^p\ast \partial \Delta^q$ by $\partial\Delta^p \ast \Delta^q$, $p+q=2d$. We have to show that $\complex{S}'$ satisfies $(\ast)$ as well. Let $M$ be a missing $d$-face of $\complex{S}'$ and let $\complex{L}:=\complex{S}'_{\leq d}\cup\{M\}$.

If $M$ is also a missing face of $\complex{S}$ then, by Observation~\ref{obs:PLspheres-are-good}, the complexes $\K$ and $\complex{L}$ satisfy all assumptions of  Theorem~\ref{thm:bistellar-moves}. By the induction hypothesis, $\obstwod(\K)\neq 0$, and thus $\obstwod(\complex{L})\neq 0$ as well. Thus, we may assume that $M$ is not a missing face of $\complex{S}$. This means that either $M \in \complex{S}_d$, or there exists a facet $F$ of $M$ with $F\not \in \complex{S}_{d-1}$.
\begin{claim}
Either $p=q=d$ and $M=\Delta^p$, or $q=d-1$ and $M=v_0\ast \Delta^q$ for some
$v_0\notin \Delta^p_0 \cup \Delta^q_0$.
\end{claim}
\begin{proof}
If $M\in \complex{S}$, then $M$ has to be deleted during the bistellar move, i.e., $M\supseteq \Delta^p$. On the other hand, all facets of $M$ have to be preserved, hence $M=\Delta^p$. Otherwise, if $M\not \in \complex{S}$, then also $F \not \in \complex{S}$ for some facet $F \prec M$. On the other hand, $F \in \complex{S}'$, hence $F =G \ast \Delta^q$ for some $G \in (\partial \Delta^p)_{d-q-2}$, since these are the only faces that are newly created. In particular, $q\leq d-1=\dim F$ and $p=2d-q\geq d+1$. Consequently, $M=v_0\ast G \ast \Delta^q$ for some vertex $v_0 \not \in \Delta^q_0$.

We have $\dim (v_0 \ast G)=d-q-1=p-d-1<p$. Thus, $v_0\notin \Delta^p_0$, else $v_0 \ast G \in \partial \Delta^p$ (and hence $M\in \complex{S}'$, a contradiction). It follows that if $G\neq \emptyset$, then for any facet $H$ of $G$, $v_0\ast H \ast \Delta^q$ is a facet of $M$ that is not a face of $\complex{S}'$, contradicting the assumption that $M$ is a missing face. Thus, $M=v_0\ast \Delta^q$.
\end{proof}

In the first case, if $p=q=d$ and $M=\Delta^p$, then $\Delta^q$ is a missing face of $\complex{S}$ and $\complex{L}=\complex{K}$, so we are again done by the induction hypothesis.

Otherwise, $M=\Delta^q\cup \{v_0\}$ for $v_0\notin \Delta^p_0 \cup \Delta^q_0$. We define a cycle $\omega\in C_{2d}(\delprod{\complex{L}}/\Z_2; \Z_2)$ that will be a witness cycle for the nonvanishing of the van Kampen obstruction.
Let $V:=\{v_0\}\cup \Delta^p_0 \cup \Delta^q_0$. Observe that $\Delta^p_0 =(\partial \Delta^p)_0$, since $p\geq 1$, hence $V\subseteq \complex{L}_0$. Moreover, $|V|=2d+3$; if we knew that all $d$-simplices $\sigma \subseteq V$ are faces of $\complex{L}$, they would form a copy of the van Kampen complex
$\Delta^{2d+2}_{\leq d}$, and we could take $\omega$ to consist of all vertex-disjoint pairs $\{\sigma, \tau\}$ of such simplices. Instead, we will work with a ``homological version'', as follows: For $\sigma \subseteq V$, $\dim \sigma =d$, define a $d$-chain $B(\sigma)$ in $\complex{L}$ as follows: $B(M)=M$, $B(\sigma)=\sigma$ for $\sigma \subseteq \Delta^p_0 \cup \Delta^q_0$ (these simplices are certainly in $\complex{L}$), and $B(\sigma)=v_0\bullet \rho$ for $\sigma=v_0 \ast \rho$, where $\bullet$ is the homological coning operation defined in Lemma~\ref{lem:homological-coning}; recall that if $\sigma \in \complex{L}$ then $v_0 \bullet \rho= \sigma$. In any case, $\rho \in \K\cap \complex{L}$, thus $v_0 \bullet \rho$ is well defined. Let
$$\omega=\sum\{\{B(\sigma),B(\tau)\}: \sigma,\tau \subseteq V ,\ \dim \sigma=\dim \tau=d,\ \sigma\cap \tau=\emptyset \}.
$$
By Lemma~\ref{lem:homological-coning} all pairs in the above
sum are indeed disjoint, and clearly all simplices occuring in these pairs are in $\complex{L}$,
hence $\omega\in C_{2d}(\delprod{\complex{L}}/\Z_2; \Z_2)$. We will
show that $\omega $ is a witness cycle for the obstruction cocycle;
first let us verify that it is a cycle:

$M$ is paired up with $\alpha_M=\{\tau\in \complex{L}:
\{M,\tau\}\in \omega\}$ in $\omega$. As $|V\setminus M|=d+2$,
every $d$-element subset of $V \setminus M$ is contained in exactly two
$d$-simplices $\tau \in \alpha_M$, hence $\partial \alpha_M = 0$. Similarly, for all simplices $\sigma=v_0\ast \rho$ and all simplices $\sigma'\in v_0\bullet \rho$, we have $\partial \alpha_\sigma=0$ and $\partial \alpha_{\sigma'} = 0$.

For $\sigma \subseteq U:=V\setminus\{v_0\}$, $\partial \alpha_\sigma =\partial(U\setminus \sigma +
\sum_{u\in U\setminus \sigma }v_0\bullet((U\setminus \sigma)\setminus \{u\}))=
\sum_{u\in U\setminus \sigma}(U\setminus \sigma)\setminus \{u\} + \sum_{u\in
U\setminus \sigma}[(U\setminus \sigma)\setminus \{u\}+\sum_{w\in (U\setminus
\sigma)\setminus \{u\}}v_0\bullet((U\setminus \sigma)\setminus \{u,w\})]=0$,
where the middle equation follows from the property (\ref{eq:coning}) established in the proof of Lemma~\ref{lem:homological-coning}.

To see that $\obstwod(\complex{L})(\omega)\neq 0$, consider a point $o\in \topint \|\complex{S}'\| \setminus \|\complex{B}'\cup \complex{S}'_{\leq d}\|$, where $\complex{B}'= \partial \Delta^p \ast \Delta^q$, and a homeomorphism $h: \|\complex{S}'\|\setminus \{o\}\rightarrow \R^{2d}$.
Notice that in the restriction of $h$ to $\complex{L}\setminus \{M\} =\complex{S}'_{\leq d} \hookrightarrow \R^{2d}$, the cycles $\|\partial M\|$ and $\|\partial(V\setminus M)\|=\|\partial \Delta^p\|$ are linked (over $\Z_2$).
To see this, consider the convex geometric embedding of $\complex{B}' = \partial \Delta^p \ast \Delta^q$ in $\R^{2d}$ obtained by placing the vertices of $\Delta^p$ and of $\Delta^q$ in complementary subspaces of $\R^{2d}$ (recall that $p+q=2d$) such that the interiors of the two simplices intersect in a single point and extending linearly on simplices. On the other hand, we have the embedding of $\complex{B}'$ given by $h$. Let $g$ be the homeomorphism from $h(\|\complex{B}'\|)$ to the geometric realization of $\complex{B}'$. If we extend the geometric realization to $\Delta^p$, then the image of $\Delta^p$ will be contained in the convex embedding of $\complex{B}'$ and there will be exactly one
transversal point of intersection between the interiors of $\Delta^p$ and $\Delta^q$. Thus, by pulling back this convex geometric realization of $\Delta^p$ using $g$, we can extend $h$ to $\Delta^p$ in such a way that $h(\Delta^p)$ is contained in $h(\complex{B}')$ and such that $h(\Delta^p)\cdot h(\Delta^q)=1$. Since $h(\partial M)\setminus h(\Delta^q)$ is disjoint from $h(\complex{B}')$, it follows that also $h(\Delta^p)\cdot h(\partial M)=1$.

By the relation between linking and intersection numbers mentioned in Section~\ref{sec:Background}, it follows that if we take a general position extension of $h$ to $M$ instead of to $\Delta^p$, the intersection number $h(M)\cdot h(\partial \Delta^p)\neq 0 \imod{2}$. Moreover, in the cycle $\omega$, the simplex $M$ is paired exactly with $\partial \Delta^p$, and $h$ is injective on $\complex{L}\setminus \{M\}$.
It follows that $\obstwod(\complex{L})(\omega)=\varphi_h(\omega)=h(M)\cdot h(\partial \Delta^p)\neq 0 \imod{2}$.
This completes the proof of the theorem.\end{proof}

\section{Relation to the $g$-conjecture}\label{sec:g-conj}
McMullen~\cite{McMullen:NumberFaces-71} conjectured a full
characterization of the possible $f$-vectors of simplicial
polytopes, by means of numerical conditions on the face numbers.
He also asked whether this characterization holds for all
simplicial spheres, now known as the $g$-conjecture. The
sufficiency of these conditions was proved by Billera and Lee
\cite{BilleraLee:SufficiencyMcMullensConditions-1981} who
constructed boundary complexes of simplicial polytopes with the
prescribed $f$-vectors, and necessity for (the smaller family of)
boundary complexes of simplicial polytopes was proved by Stanley
\cite{Stanley:NumberFacesSimplicialPolytope-80}, using algebraic
machinery. We review now an algebraic $g$-conjecture, from which
the $g$-conjecture follows; this is what Stanley proved in the
case of simplicial polytopes by using the hard Lefschetz theorem
for toric varieties of rational polytopes. For further background,
we refer the reader to \cite{Kalai:AlgebraicShifting-02} and the
references therein.

Let $\K$ be a simplicial complex on $\K_0=[n]$. Let
$A=\R[x_1,..,x_n]$ be the polynomial ring, with the usual grading by degree.
% each variable has
%degree one.
The \emph{face ring} of $\K$ is
$\R[\K]=A/I_{\K}$ where $I_\K$ is the ideal in $A$ generated by
the monomials whose support is \emph{not} an element of $\K$. Let
$\Theta=(\theta_1,..,\theta_d)$ be generic
1-forms, i.e. their coefficients (in the basis of monomials in the $x_i$'s) are algebraically independent over $\Q$. Consider the quotient ring $\R[\K]/(\Theta)=H_0\oplus H_1\oplus...$
where the grading is induced by the degree grading in $A$, and
$(\Theta)$ is the ideal in $\R[\K]$ generated by the images
of the elements of $\Theta$ under the projection $A\rightarrow
\R[\K]$.

If $\K$ is a simplicial $(d-1)$-sphere then for every $0\leq i\leq d$, $\dim_{\R}H_i=h_i(\K):=\sum_{0\leq j\leq i}(-1)^{i-j}\binom{d-j}{i-j}f_{j-1}(\K)$ where $f_{j}(\K)=|\K_{j}|$ (then $f_{-1}(\K)=1$). Moreover, these numbers satisfy the \emph{Dehn-Sommerville relations} $h_i(\K)=h_{d-i}(\K)$.

The algebraic $g$-conjecture states that for an element $\omega\in A_1$ which is generic w.r.t. $\Theta$ (i.e., whose coefficients are algebraically independent from those of the $\theta_i$'s), the multiplication maps $\omega^{d-2i}: H_i\longrightarrow
H_{d-i}$, $m\mapsto \omega^{d-2i}m$, are isomorphisms for
every $0\leq i\leq \lfloor d/2\rfloor$.

Recently Swartz
\cite[Theorem~4.26]{Swartz:FromSpheresToManifolds} showed that in
order to prove the $g$-conjecture for a family of spheres that is
closed under links (e.g. PL spheres or homology spheres, but not
simplicial spheres), it is enough to show that for every $i\geq 1$
and every $2i$-sphere in that family, the map $\omega:
H_i\longrightarrow H_{i+1}$ is an isomorphism.

The algebraic $g$-conjecture has an interpretation in terms of
\emph{symmetric algebraic shifting}. Algebraic shifting, introduced by Kalai,
is an operator that assigns to each simplicial complex another,
canonically defined simplicial complex that is combinatorially simpler
(shifted), while a lot of important information about the original complex
(e.g., its face numbers as well as the Betti numbers) is preserved.
There are two variants of algebraic shifting, a \emph{symmetric} one
based on polynomial rings and another one based on \emph{exterior}
algebra; we refer to the survey \cite{Kalai:AlgebraicShifting-02} for
further background concerning algebraic shifting.

The symmetric case of the following
strong conjecture of Kalai and Sarkaria implies the algebraic
$g$-conjecture:
\begin{conjecture}[\protect{\cite[Conjecture 27]{Kalai:AlgebraicShifting-02}}]\label{conjKalaiSarkaria}
Let $\K$ be a simplicial complex on $n$ vertices which is embeddable
in the $(d-1)$-sphere. Then its algebraic shifting \textup{(}symmetric or exterior\textup{)}
satisfies $\Delta(\K)\subseteq \Delta(d,n)$, where $\Delta(d,n)$ is the symmetric shifting of the cyclic $d$-polytope on $n$ vertices \textup{(}its combinatorics is known\textup{)}.
\end{conjecture}
Equivalently, this conjecture states that embeddability of a
complex in the $(d-1)$-sphere implies that the sets
$T_d,...,T_{\lceil d/2\rceil}$, where
$T_{d-k}=\{k+2,k+3,...,d-k,d-k+2,d-k+3,...,d+2\},\  0\leq k\leq
\lfloor d/2\rfloor$, are \emph{not} in its shifting.

Due to Swartz' result mentioned above, in order to prove that the
$g$-conjecture holds for PL spheres, it is enough to prove the
following part of Conjecture \ref{conjKalaiSarkaria}:
\begin{conjecture}\label{weakKSconj}
If $\{d+3,d+4,...,2d+3\}\in \Delta^s(\K)$ \textup{(}equivalently, by shiftedness, $\Delta_{\leq d}^{2d+2}\subseteq \Delta^s(\K)$\textup{)} then $\K$ does not embed in the $2d$-sphere, where $\Delta^s$ denotes symmetric algebraic shifting.
\end{conjecture}
Indeed, if $\K$ is a triangulation of a $2d$-sphere then, by the Dehn-Sommerville equations, $\omega: H_d\longrightarrow H_{d+1}$ is an isomorphism iff it is onto, which happens iff when choosing a basis for $\R[\K]_{d+1}$ from the projection of the monomials of degree $d+1$ in generic variables $\theta_1,...,\theta_n\in A_1$, lexicographically in the greedy way, all monomials in this basis involve a variable from the set $\{\theta_1,...,\theta_{2d+2}\}$. By shiftedness, this happens iff $\theta_{2d+1}^{d+1}$ is not in this basis, which in turn happens iff $\{d+3,d+4,...,2d+3\}\notin \Delta^s(\K)$, which happens iff
$\Delta_d^{2d+2}\nsubseteq \Delta^s(\K)$.

Note that for the special case where the triangulated $2d$-sphere is
neighborly, Conjecture \ref{conjMissingFace} follows from Conjecture \ref{weakKSconj}.
Conjecture \ref{weakKSconj} is trivial for $d=0$ and known for $d=1$, where the assumption actually implies that $\K_{\leq 1}$ contains the complete graph on $5$ vertices as a minor \cite[Theorem 1.2]{Nevo:EmbeddabilityStresses}. Theorem \ref{thmMissingFacePL} verifies Conjecture \ref{weakKSconj} in the case where $\K$ is the $d$-skeleton of a neighborly PL $2d$-sphere union with a missing $d$-face. Indeed, in this case the assumption holds
as the number of $(d+1)$-tuples on $\K_0=[n]$, where $n\geq 2d+3$, which are not lexicographically greater or equal to $\{d+3,...,2d+3\}$ is only $|\Delta(2d+1,n)_d|$ while $\Delta(\K)$ is shifted with $|\Delta(\K)_d|=1+|\Delta(2d+1,n)_d|$;
 and the conclusion holds by Theorem \ref{thmMissingFacePL}.

\begin{remark}
In the special case where the neighborly PL $2d$-sphere is the boundary of the cyclic $(2d+1)$-polytope, one can use the notion of \emph{higher minors}, introduced in \cite{Nevo:HigherMinors}, to show that $\K$ contains $\Delta_{\leq d}^{2d+2}$ as a minor, and hence by \cite{Nevo:HigherMinors}, Corollary 1.2, $\K$ is
not embeddable in the $2d$-sphere.
\end{remark}

\bibliographystyle{amsplain}
\bibliography{topology,ubt}

\end{document}